\documentclass[11pt]{amsart}
\usepackage{amscd,amsfonts,amssymb,amsmath}
\usepackage{hyperref}
\usepackage{tikz}

\newtheorem{theorem}{Theorem}[section]
\newtheorem{cor}[theorem]{Corollary}
\newtheorem{lemma}[theorem]{Lemma}
\newtheorem{proposition}[theorem]{Proposition}

\theoremstyle{definition}

\numberwithin{equation}{subsection}
\theoremstyle{plain}

\newtheorem{conjecture}{Conjecture}

\newtheorem{problem}{Problem}

\def\Z{\mathbb Z}

\newcommand{\thmref}[1]{Theorem~\ref{#1}}
\newcommand{\lemref}[1]{Lemma~\ref{#1}}

\newcommand{\eqnref}[1]{~{\textrm(\ref{#1})}}
\numberwithin{equation}{section}

\begin{document}
\title{Palindromic Automorphisms of  Free Groups}
\author{Valeriy G. Bardakov}
\author{Krishnendu Gongopadhyay}
\author{Mahender Singh}
\date{\today}
\address{Sobolev Institute of Mathematics and Novosibirsk State University, Novosibirsk 630090, Russia}
\address{Laboratory of Quantum Topology, Chelyabinsk State University, Brat'ev Kashirinykh street 129, Chelyabinsk 454001, Russia.}
\email{bardakov@math.nsc.ru}
\address{Indian Institute of Science Education and Research (IISER) Mohali, Sector 81,  S. A. S. Nagar, P. O. Manauli, Punjab 140306, India.}
\email{krishnendu@iisermohali.ac.in}
\address{Indian Institute of Science Education and Research (IISER) Mohali, Sector 81,  S. A. S. Nagar, P. O. Manauli, Punjab 140306, India.}
\email{mahender@iisermohali.ac.in}


\begin{abstract}
Let $F_n$ be the free group of rank $n$ with free basis $X=\{x_1,\dots,x_n \}$. A palindrome is a word in $X^{\pm 1}$ that reads the same backwards
as forwards. The palindromic automorphism group $\Pi A_n$ of $F_n$ consists of those automorphisms that map each $x_i$ to a palindrome.
In this paper, we investigate linear representations of $\Pi A_n$, and prove that $\Pi A_2$ is linear. We obtain conjugacy classes of involutions in $\Pi A_2$, and investigate residual nilpotency of $\Pi A_n$ and some of its subgroups. Let $IA_n$ be the group of those automorphisms of $F_n$ that act trivially on the abelianisation, $P I_n$ be the palindromic Torelli group of $F_n$, and $E \Pi A_n$ be the elementary palindromic automorphism group of $F_n$. We prove that $PI_n=IA_n \cap E \Pi A_n'$. This result strengthens a recent result of Fullarton \cite{Fullarton}.
\end{abstract}
\maketitle

\section{Introduction}
Let $F_n$ be the free group of rank $n$ with free basis $X=\{x_1,\dots,x_n \}$, and let $Aut(F_n)$ be the automorphism group of $F_n$. A reduced word $w=x_1^{\epsilon_1}\dots x_n^{\epsilon_n}$ in $X^{\pm1}$ is called a palindrome if it is equal to its reverse word  $\overline{w}=x_n^{\epsilon_n} \dots x_1^{\epsilon_1}$. In \cite{Collins}, Collins defined the palindromic automorphism group $\Pi A_n$  as the subgroup of $Aut(F_n)$ consisting of those automorphisms that map each $x_i$ to a palindrome. He proved that $\Pi A_n$ is finitely presented, and that it is generated by the following three types of automorphisms
$$
t_{i} : \left\{
\begin{array}{ll}
x_i \longmapsto x_i^{-1} &  \\
x_k \longmapsto x_k &  \textrm{for}~k \neq i,
\end{array} \right.
$$

$$
\alpha_{i, i+1} : \left\{
\begin{array}{ll}
x_i \longmapsto x_{i+1} &  \\
x_{i+1} \longmapsto x_i & \\
x_k \longmapsto x_k & \textrm{for}~ k \neq i,
\end{array} \right.
$$

$$
\mu_{ij} : \left\{
\begin{array}{ll}
x_i \longmapsto x_jx_ix_j & \textrm{for}~i \neq j \\
x_k \longmapsto x_k & \textrm{for}~k \neq i.
\end{array} \right.
$$

The group $$E\Pi A_n= \langle \mu_{ij}~|~1 \leq i \neq j \leq n\rangle$$ is called the elementary palindromic automorphism group of $F_n$, and
the group $$ES_n= \langle t_i, \alpha_{j,j+1} ~|~\ 1 \leq i \leq n, 1 \leq j \leq n-1 \rangle$$ is called the extended symmetric group. In \cite{Collins}, Collins showed that $$\Pi A_n \cong E\Pi A_n \rtimes ES_n$$ for $n \geq 2$. Here, $ES_n$ acts on $E \Pi A_n$ by conjugation given by the following rules
$$t_i \mu_{ij} t_i=\mu_{ij}^{-1}, ~ t_j \mu_{ij} t_j=\mu_{ij}^{-1},$$
$$t_k \mu_{ij} t_k=\mu_{ij}~\textrm{for}~k \neq i, j,$$
$$\alpha \mu_{ij} \alpha=\mu_{\alpha(i) \alpha(j)}~\textrm{for}~\alpha \in \{\alpha_{1,2}, \alpha_{2,3}, \dots, \alpha_{n-1, n} \}.$$

In \cite{Collins}, Collins also showed that a set of defining relations for $E\Pi A_n$ is
$$\mu_{ij}\mu_{kl}= \mu_{kl}\mu_{ij},$$
$$\mu_{ik}\mu_{jk}= \mu_{jk}\mu_{ik},$$
$$\mu_{ik}\mu_{jk}\mu_{ij}= \mu_{ij}\mu_{jk}\mu_{ik}^{-1}.$$

In the same paper, Collins conjectured that $E\Pi A_n$ is torsion free for each $n \geq 2$. Using geometric techniques, Glover and Jensen  \cite{Glover} proved this conjecture and also calculated the virtual cohomological dimension of $\Pi A_n$. Using methods from logic theory, Piggott and Ruane \cite{pr} constructed Markov languages of normal forms for $\Pi A_n$. The pure palindromic automorphism group is defined as
$$P \Pi A_n= \langle \mu_{ij}, t_k~|~1 \leq i \neq j \leq n~\textrm{and}~1 \leq k \leq n \rangle.$$
Fairly recently, Jensen, McCammond and Meier \cite{jemm} computed the Euler characteristic of $\Pi A_n$, $P \Pi A_n$ and $E \Pi A_n$. Nekritsuhin \cite{ne2, ne3} has computed the center of $\Pi A_n$ using combinatorics of words in $F_n$. The palindromic Torelli group $P I_n$ is the subgroup of $E \Pi A_n$ consisting of those automorphisms that act trivially on the abelianization of $F_n$. The group $P I_n$ is a free group analogue of the hyperelliptic Torelli group of the mapping class group of an oriented surface. In his PhD Thesis \cite{Fullarton1}, Fullarton has obtained a normal generating set for $P I_n$ by constructing a simplicial complex on which $P I_n$ acts in a nice manner. In this paper, we investigate some fundamental group theoretic questions about $\Pi A_n$ generalising some results of Nekritsuhin and strengthening a result of Fullarton.

Throughout the paper, we use standard notation and convention. All functions are evaluated from left to right. If $G$ is a group, then $G'$ denotes the commutator subgroup of $G$, $Z(G)$ denotes the center of $G$, and $\gamma_n(G)$ denotes the $n$th term in the lower central series of $G$. Given two elements $g$ and $h$ in $G$, we denote the element $h^{-1}gh$ by $g^h$ and the commutator $g^{-1}h^{-1}gh$ by $[g,h]$.

First, in Section \ref{sec2}, we determine the center of $\Pi A_n$. In \cite{ne}, Nekritsuhin observed that there are five involutions in $\Pi A_2$ that are not conjugate to each other. It is natural to ask for the precise number of conjugacy classes of involutions of $\Pi A_2$.  In Section \ref{sec3}, we show that any involution in $\Pi A_2$ is conjugate to one of the five involutions in Nekritsuhin's list. In Section \ref{sec4}, we consider the generalization of Sanov representation that Nekritsuhin constructed in \cite{ne}.  We show that the Nekritsuhin representation is faithful. A precise description of the image is also obtained. We also prove that $\Pi A_2$ is linear. In Section \ref{sec5}, we prove that $\Pi A_n$ is not residually nilpotent  for $n \geq 3$. Also, we construct a  subgroup $\widetilde{ U_n}$ of $E \Pi A_n$ that is residually nilpotent for $n=3$. Finally, in Section \ref{sec6}, we prove that  $P I_n= IA_n \cap E \Pi A_n'$. This result strengthens a recent result of Fullarton \cite{Fullarton}.  
\bigskip

\section{Center of $\Pi A_n$}\label{sec2}
In this section, we compute the center of $\Pi A_n$ for all $n \geq 2$. In \cite{ne2}, Nekritsuhin computed the center of $\Pi A_2$ using combinatorics of words in $F_2$. We generalised this result and computed the center of $\Pi A_n$ for $n \geq 3$. But, when this paper was being written, we came to know that Nekritsuhin has also proved this result in \cite{ne3}. We present a proof here for the sake of completeness.

\begin{theorem}\label{center} 
The center of  $\Pi A_n$ is the cyclic group of order two generated by the automorphism
$$t: x_i \mapsto x_i^{-1} ~\textrm{for all}~ 1 \leq i\leq  n.$$
\end{theorem}

\begin{proof}
It is easy to check that $t$ lies in $Z(\Pi A_n)$. Suppose that $\phi$ lies in $Z(\Pi A_n)$. We claim that $\phi=t^{\epsilon}$ for $\epsilon=0$ or $1$.  Let $1 \leq i \leq  n$ be fixed. Suppose that 
$$ (x_i)\phi=x_i ^{m_{i_1}} u_{i_1} x_i^{m_{i_2}} u_{i_2} \ldots x_i^{m_{i_k}}u_{i_k} x_i^{m_{i_{k +1}}},$$
where $m_{i_r} \in \Z$ and $u_{i_r} \in F(x_1, \ldots, \hat x_i, \ldots, x_n )$, the free group on $n-1$ generators. Since $(x_i)\phi$ is a palindrome, it follows that $m_{i_1}=m_{i_{k +1}}$, $u_{i_1} ={ \overline{u}_{i_k}}$, $m_{i_2}=m_{i_{k}}$, $u_{i_2} ={ \overline{u}_{i_{k-1}}}$, and so on.

Since $\phi$ is central, we have $\phi t_i= t_i \phi$. Hence 
$$(x_i)t_i \phi = x_i^{-m_{i_{k+1}}} u_{i_k}^{-1} \ldots u_{i_1}^{-1} x_i^{-m_{i_1}}$$
and
$$(x_i)\phi t_i=x_i^{-m_{i_1}} u_{i_1} x_i^{-m_{i_2}} \ldots u_{i_k} x_i^{-m_{i_{k+1}}}$$
are equal. This implies that $u_{i_k}^{-1} =u_{i_1}$. But, since $u_{i_1}=\overline{u}_{i_k}$, we obtain that $u_{i_1}=1$. Continuing this way, we get that $\phi(x_i)=x_i^{m_i}$ for all $1 \leq i \leq n$. 

Again, since $\phi$ is central, we have $\phi \alpha_{i, i+1}=\alpha_{i, i+1} \phi$. It is easy to see that 
$$\phi \alpha_{i, i+1} : \left\{ \begin{array}{ll}
x_i \longmapsto x_{i+1}^{m_i} & \\
x_{i+1} \longmapsto x_i^{m_{i+1}}& \end{array} \right.
$$
and
$$
\alpha_{i, i+1} \phi: \left\{ \begin{array}{ll}
x_i \longmapsto x_{i+1}^{m_{i+1}} & \\
x_{i+1} \longmapsto x_i^{m_i}. & \end{array} \right. $$
This implies that $m_i=m_{i+1}$ and hence $m_1=m_2=\cdots=m_n=m$, say. Further, since $\phi$ is a central automorphism, it follows from the short exact sequence \eqnref{se1} that projection of $\phi$ is an element of the center of ${\rm GL}(n, \Z)$. This implies that $m=1$ or $-1$ and hence $\phi=t^{\epsilon}$.  This proves the theorem. 
\end{proof}

As a consequence, we have the following corollary.

\begin{cor}\label{cor-c}
The center of $E \Pi A_n$ is trivial. 
\end{cor}
\bigskip

\section{Conjugacy classes of involutions in $\Pi A_2$} \label{sec3}
It is well known that there are four conjugacy classes of involutions in $Aut(F_2)$. See for example \cite[Section 4, Proposition 4.6]{Lyndon}. It is interesting to find conjugacy classes of involutions in $\Pi A_2$ since $\Pi A_2 \leq Aut(F_2)$. First, we investigate involutions in $\Pi A_2$. For simplicity, we denote the elements $t_1$, $t_2$ and $\alpha_{12}$ of $\Pi A_2$ by $\sigma$, $\delta$ and $\rho$, respectively. Let  $F_2=\langle x, y\rangle$ and $E \Pi A_2=\langle {\mu_{12}}, {\mu_{21}} \rangle$,  where
$${\mu_{12}}: \left\{ \begin{array}{ll}
x \longmapsto yxy & \\
y \longmapsto y &, \end{array} \right.
{\mu_{21}}: \left\{ \begin{array}{ll}
x \longmapsto x & \\
y \longmapsto xyx. & \end{array} \right. $$

We have five involutions in $E S_2$ given by the automorphisms $\sigma, ~ \delta, ~ \rho, ~ \sigma \delta$ and $\sigma \delta \rho$,  where

$$\sigma: \left\{
\begin{array}{ll}
x \longmapsto x^{-1} & \\
y \longmapsto y &,
\end{array} \right.
\delta: \left\{
\begin{array}{ll}
x \longmapsto x & \\
y \longmapsto y^{-1} &,
\end{array} \right.
\rho: \left\{
\begin{array}{ll}
x \longmapsto y & \\
y \longmapsto x. &
\end{array} \right.
$$

We are interested in obtaining involutions in $\Pi A_2=E \Pi A_2 \rtimes ES_2$. Note that $E \Pi A_2 =\langle \mu_{12}, \mu_{21} \rangle \cong F_2$. By a palindrome in  $E \Pi A_2$, we mean a palindrome in the free generating set $\{ \mu_{12}, \mu_{21} \}$ of $E \Pi A_2$.

\begin{lemma} \label{inv1}
Let $\lambda \in \{\sigma, \delta\}$ and $f \in E \Pi A_2$. Then $f \lambda$ is an involution if and only if  $f$ is a palindrome in $E \Pi A_2$.
\end{lemma}

\begin{proof}
Let $f={\mu_{12}}^{a_1} {\mu_{21}}^{b_1} \ldots {\mu_{12}}^{a_k} {\mu_{21}}^{b_k}$ for some integers $a_i$, $b_i$. Then
$$(f \lambda)^2=f \lambda f \lambda= f f^{\lambda} \in E \Pi A_2.$$
We know that
$${\mu_{12}}^{\sigma}={\mu_{12}}^{-1}, ~ {\mu_{21}}^{\sigma}={\mu_{21}}^{-1},$$
$${\mu_{12}}^{\delta}={\mu_{12}}^{-1}, ~ {\mu_{21}}^{\delta} ={\mu_{21}}^{-1}.$$
It follows that 
$$ff^{\lambda}={\mu_{12}}^{a_1} {\mu_{21}}^{b_1} \ldots {\mu_{12}}^{a_k} {\mu_{21}}^{b_k} ~ {\mu_{12}}^{-a_1} {\mu_{21}}^{-b_1} \ldots {\mu_{12}}^{-a_k} {\mu_{21}}^{-b_k}.$$
If $f$ is a palindrome, then we have the following two cases.
\begin{enumerate}
\item $b_k=0$ and $a_k=a_1$, $b_{k-1}=b_1$, $a_{k-1}=a_2$ and so on. This implies that $ff^{\lambda}=1$.
\item  $a_1=0$ and $b_k=b_1$, $a_k=a_2$ and so on. This again implies that $ff^{\lambda}=1$.
\end{enumerate}
Conversely, if $ff^{\lambda}=1$, then we reduce to the above two cases and that implies that $f$ is a palindrome.
\end{proof}

\begin{lemma}\label{inv2}
Let $f \in E \Pi A_2$. Then $f \sigma \delta$ is an involution if and only if $f=1$.
\end{lemma}
\begin{proof}
If $f \sigma \delta$ is an involution, then $f \sigma \delta f \sigma \delta=1$. By Theorem \ref{center}, $\sigma \delta \in Z(\Pi A_2)$. It follows that $f=1$. The converse is obvious.
\end{proof}

\begin{lemma}\label{inv3}
Let $\lambda \in \{\rho, \sigma \delta \rho\}$ and  $f={\mu_{12}}^{a_1} {\mu_{21}}^{b_1} \ldots {\mu_{12}}^{a_k} {\mu_{21}}^{b_k}$ be an element in $E \Pi A_2$. Then $f \lambda$ is an involution if and only if
\begin{equation} \label{fp} f=
\left\{\begin{array}{ll}
{\mu_{12}}^{a_1} {\mu_{21}}^{b_1} \ldots {\mu_{12}}^{a_m} {\mu_{21}}^{b_m} {\mu_{12}}^{-b_m} {\mu_{21}}^{-b_m} \ldots {\mu_{12}}^{-b_1} {\mu_{21}}^{-a_1} \hbox{ if } k=2m & \\\
 &\\
{\mu_{12}}^{a_1} {\mu_{21}}^{b_1} \ldots {\mu_{12}}^{a_m} {\mu_{21}}^{b_m} ({\mu_{12}}^{a_{m+1}} {\mu_{21}}^{-a_{m+1}}) {\mu_{12}}^{-b_m} {\mu_{21}}^{-a_m} \ldots {\mu_{12}}^{-b_1} {\mu_{21}}^{-a_1} \hbox{ if }k=2m+1. &
 \end{array} \right.\end{equation}
\end{lemma}

\begin{proof}
The proof is obtained by induction. We sketch the proof for $\lambda=\rho$. The other cases are similar. We have
$$(f\rho)^2=ff^{\rho}={\mu_{12}}^{a_1} {\mu_{21}}^{b_1} \ldots {\mu_{12}}^{a_k} {\mu_{21}}^{b_k} ~ {\mu_{21}}^{a_1} {\mu_{12}}^{b_1} \ldots {\mu_{21}}^{a_k} {\mu_{12}}^{b_k}.$$
Since $(f\rho)^2=1$, we must have $a_l=-b_{k-l+1}, ~ b_l=-a_{k-l+1}$ for $l=1, \ldots, k$. Now it follows that $f$ has the form given in the lemma according as $k=2m$ or $k=2m+1$.
\end{proof}

We consider the following problem.

\begin{problem}
Is the conjugacy problem decidable in $\Pi A_n$ for $n \geq 3$.
\end{problem}

For $n=2$, the answer is positive and is given in the following lemma.

\begin{lemma}\label{con}
Let $g, g' \in \Pi A_2$ be such that $g=f \lambda$ and $g'=f' \lambda'$, where $f, f' \in E \Pi A_2$ and $\lambda, \lambda' \in E S_2$. Then $g$ and $g'$ are conjugate in $\Pi A_2$ if and only if
\begin{enumerate}
\item $\lambda'=\lambda^{ \mu}$ for some $\mu \in ES_2$.
 \item  $(f')^{\mu^{-1}}=l^{-1} f ~ l^{\lambda^{-1}}$ for some $l \in E \Pi A_2$.
\end{enumerate}
\end{lemma}

\begin{proof}
$g$ and $g'$ are conjugate if and only if there exists $h=l \mu$ in $\Pi A_2$ such that
\begin{eqnarray*}
f' \lambda'=h^{-1}  g h &=& (l\mu)^{-1}~ f \lambda~ l \mu\\
&=& \mu^{-1} l^{-1} ~ f  ~ l ^{\lambda^{-1}} \lambda \mu \\
&=& (\mu^{-1} l^{-1} \mu)(\mu^{-1} f \mu) \mu^{-1} \lambda l \lambda^{-1} \mu \mu^{-1} \lambda \mu \\ &=& \big( (l^{-1})^{\mu } f ^{\mu} l^{\lambda^{-1} \mu} \big) (\lambda^{ \mu}).
\end{eqnarray*}
Thus $f'=(l^{-1})^{\mu } f ^{\mu} l^{\lambda^{-1}}$ and $\lambda'=\mu^{-1} \lambda \mu$. This proves the lemma.
\end{proof}

In \cite{ne2}, Nekritsuhin proved the following result.

\begin{theorem}
The following involutions are not conjugate in $\Pi A_2$
$$\sigma, ~ \sigma \delta,~\rho, ~ {\mu_{12}} \sigma, ~ {\mu_{12}} \delta=\rho({\mu_{21}} \sigma)\rho.$$
\end{theorem}

\begin{lemma}\label{nel1}
There are three conjugacy classes of involutions in $ES_2$ represented by $\sigma, ~ \rho, ~ \sigma \delta.$
\end{lemma}
\begin{proof}
This follows from the fact that $\rho \sigma \rho=\delta$ and $\sigma(\rho \sigma \delta) \sigma=\sigma \rho \delta=\rho(\rho \sigma \rho) \delta=\rho$.
\end{proof}

If $f \lambda$ and $f' \lambda'$ are two involutions that are conjugate in $\Pi A_2$, then it follows from \cite{ne} that $\lambda=\lambda'$. With this, we have the following result.

\begin{lemma} \label{inv4}
Let $f \lambda$ be an involution in $\Pi A_2$, where $f \in E \Pi A_2$ and $\lambda \in \{\sigma, \delta\}$. Then $f \lambda$ is conjugate to one of the involutions $\sigma$, ${\mu_{12}} \sigma$, ${\mu_{21}} \sigma$.
\end{lemma}

Note that, the conjugation rule  in $\Pi A_2$ is
$$\sigma^{\rho} = \delta,~ ( \mu_{12} \sigma)^{\rho} = \mu_{12} \delta, ~( \mu_{21} \sigma)^{\rho} = \mu_{21} \delta.$$ Thus an equivalent formulation of the above lemma is the following.
 
\begin{lemma} \label{inv5}
Let $f \lambda$ be an involution in $\Pi A_2$, where $f \in E \Pi A_2$ and $\lambda \in \{\sigma, \delta\}$. Then $f \lambda$ is conjugate to one of the involutions $\delta$, ${\mu_{12}} \delta$, ${\mu_{21}}  \delta$.
\end{lemma}

\begin{proof}
Let $f \lambda$ and $g \lambda$ be conjugate involutions, where $\lambda \in \{\sigma, \delta\}$.  We want to obtain a representative $(f \lambda)^h$ of the conjugacy class. Using  \lemref{inv1} and \lemref{con}, without loss of generality, we can assume that
$$f={\mu_{12}}^{a_1} {\mu_{21}}^{b_1} \ldots {\mu_{21}}^{b_1 }{\mu_{12}}^{a_1} \lambda $$ and $$h={\mu_{12}}^{c_1} {\mu_{21}}^{d_1} \ldots {\mu_{12}}^{c_k} {\mu_{21}}^{d_k}.$$
Then
\begin{equation*}
\begin{split}
(f \lambda)^h = ~& {\mu_{21}}^{-d_k} {\mu_{12}}^{-c_k} \ldots {\mu_{21}}^{-d_1} {\mu_{12}}^{-c_1} ~ {\mu_{12}}^{a_1} {\mu_{21}}^{b_1} \\
& \ldots {\mu_{21}}^{b_1 }{\mu_{12}}^{a_1} ~ {\mu_{12}}^{-c_1} {\mu_{21}}^{-d_1}\ldots {\mu_{12}}^{-c_k} {\mu_{21}}^{-d_k} \lambda.
\end{split}
\end{equation*}
We can choose $h$ such that $c_1=-a_1$, $d_1=-b_1$, etc. This reduces the length of the word
$(f \lambda)^h$ and after cancellation of the middle terms, we get a reduced word of the form
$${\mu_{12}}^{-l} {\mu_{12}}^a {\mu_{12}}^{-l}={\mu_{12}}^{a-2l}$$
or 
$${\mu_{21}}^{-l} {\mu_{21}}^b {\mu_{21}}^{-l}={\mu_{21}}^{b-2l}.$$ Depending on whether $a$ is even or odd, we can choose $l$ such that $a-2l=1$ or $0$, and we can do the same for $b-2l$. Hence the conjugacy representatives would be the desired ones.
\end{proof}

\begin{lemma}\label{inv7}
Any involution of the form $f\rho$ where $f \in E \Pi A_2$ is conjugate to $\rho$ by some element of $E \Pi A_2$.
\end{lemma}
\begin{proof}
Take an involution of the form $f\rho$, where $f$ is of the form \eqnref{fp}. After writing down the expression for $(f \rho)^h$, as in the proof of \lemref{inv5}, we can choose $h$ such that $(f\rho)^h=\rho$.
\end{proof}

Combining the above lemmas, we establish the following theorem.

\begin{theorem}
There are five conjugacy classes of involutions in $\Pi A_2$. These conjugacy classes are represented by
$$\sigma, ~ \sigma \delta, ~ \rho, ~ {\mu_{12}} \sigma, ~ {\mu_{21}} \sigma.$$
\end{theorem}
\bigskip

\section{Linear Representation of $\Pi A_n$}\label{sec4}
In \cite{ne}, Nekritsuhin studied the linear representation $\phi: \Pi A_2 \to {\rm GL}(2, \Z)$ given by
$$(\mu_{12})\phi=\begin{pmatrix} 1 & 2 \\ 0 & 1 \end{pmatrix}, ~ (\mu_{21})\phi=\begin{pmatrix} 1 & 0 \\ 2 & 1 \end{pmatrix},$$
$$(\sigma)\phi=I_1=\begin{pmatrix} -1 & 0 \\ 0 & 1 \end{pmatrix}, ~ (\delta)\phi=I_2=\begin{pmatrix} 1 & 0 \\ 0 & -1 \end{pmatrix}, ~ (\rho)\phi=I_{12} =\begin{pmatrix} 0 & 1 \\ 1 & 0 \end{pmatrix}.$$

This representation comes from the natural homomorphism $Aut(F_n) \to \rm GL(n, \mathbb{Z})$. We prove the following result.

\begin{proposition}\label{repfaith}
The representation $\phi$ is faithful.
\end{proposition}

\begin{proof}
We know that
$$ES_2=\{ 1, \sigma, \delta, \sigma \delta, \rho, \sigma \rho, \delta \rho, \sigma \delta \rho \}.$$
Consider the $\phi$ images of elements of $ES_2$.
It is easy to see that
$$I_1 I_2=\begin{pmatrix} -1 & 0 \\ 0 & -1 \end{pmatrix}, ~ I_1 I_{12} = \begin{pmatrix} 0 & -1 \\ 1 & 0 \end{pmatrix},$$
$$I_2 I_{12} = \begin{pmatrix} 0 & 1 \\ -1 & 0 \end{pmatrix}, ~I_1 I_2 I_{12} = \begin{pmatrix} 0 & -1 \\ -1 & 0 \end{pmatrix}.$$
Thus $\phi|_{ES_2}$ is faithful. Further, the representation $\phi|_{E \Pi A_2}$ is the classical Sanov representation that is well-known to be faithful \cite{sa}.

Now, consider an arbitrary element $f \lambda \in \Pi A_2$, where $f \in E \Pi A_2$ and $\lambda \in ES_2$. Suppose that $(f \lambda)\phi=1$. But then $(f)\phi=1$ and $(\lambda)\phi=1$. By the above two cases, $f=1$ and $\lambda=1$. Hence kernel of $\phi$ is trivial. This proves the result.
\end{proof}

We are unable to answer the following interesting problem.

\begin{problem}
Is $\Pi A_n$ linear for $n \geq 3$?
\end{problem}

Given a matrix $M$ in ${\rm GL}(2, \Z)$, we would like to understand the conditions under which $M$ lies in the image of $\phi$. The following result answers this question.

\begin{proposition}
A matrix $M$ in ${\rm GL}(2, \Z)$ lies in the image of $\phi$ if and only if it has one of the following eight forms
$$\pm \begin{pmatrix} 1 + 4a_{11} &   2a_{12} \\  2a_{21} &  1+4a_{22} \end{pmatrix}, 
\pm \begin{pmatrix}  1 + 4a_{11} &  -2a_{12} \\ 2a_{21} &  -1-4a_{22} \end{pmatrix},$$
$$\pm \begin{pmatrix} 2a_{12} & -1-4a_{11} \\ 1 +4a_{22} & - 2a_{21}\end{pmatrix}, \ 
\pm \begin{pmatrix} 2a_{12} & 1+4a_{11} \\ 1 +4a_{22} & 2a_{21}\end{pmatrix},$$
for some integers $a_{11}$, $a_{12}$, $a_{21}$ and $a_{22}$ satisfying $a_{11} + a_{22}-a_{12} a_{21} +4 a_{11} a_{22}=0$.
\end{proposition}

\begin{proof}
It follows from \cite{ne} that if $M \in (\Pi A_2)\phi$, then every row of $M$ has one even element and one odd element. One can prove by induction that, any matrix in $(E \Pi A_2)\phi$ has the form
$$M_o=\begin{pmatrix} 1 + 4 a_{11} & 2 a_{12} \\
2 a_{21} & 1 + 4 a_{22} \end{pmatrix} \in {\rm SL}(2, \Z)$$  for some integers $a_{11}$, $a_{12}$, $a_{21}$ and $a_{22}$.
In particular, $(1 + 4 a_{11})(1 + 4 a_{22})-4a_{12} a_{21}=1$, that is,
$$a_{22} + a_{11} -a_{12} a_{21} + 4 a_{11} a_{22} =0.$$
An arbitrary matrix $M \in (\Pi A_2)\phi$ has the form $M=M_o S$ where
$$ S \in (ES_2)\phi=\{ I, I_1, ~ I_2, ~ I_1 I_2, ~ I_{12}, ~ I_1 I_{12},~  I_2 I_{12}, ~ I_1 I_2 I_{12} \}.$$ Here $I$ denotes identity matrix. Now computing $M_o S$ for each  $S \in (ES_2)\phi$ gives the desired result.
\end{proof}

Recall that, the Sanov representation of free group $F_2=\langle x, y \rangle$ is given by $$x \mapsto \begin{pmatrix} 1 & 2 \\ 0 & 1 \end{pmatrix}, ~ y \mapsto \begin{pmatrix} 1 & 0 \\ 2 & 1 \end{pmatrix}.$$ It is a natural problem whether this can be generalized to a representation for free group $F_n$ of rank $n\geq 3$.

For $n \geq 1$ and $a \in \mathbb{Z}$, let $t_{ij}(a) \in {\rm SL}(n,  \Z)$ be the matrix whose $(i, j)$th entry is $a$, all diagonal entries are one and all other entries are zero.  Let
$$H_n^a= \langle t_{ij}(a)~|~1 \leq i \neq j \leq n\rangle \leq {\rm SL}(n, \mathbb{Z}).$$
Note that, if $a=2$, then $H_n^2$ is the image of $E \Pi A_n$ under the natural epimorphism from $Aut(F_n)$ to $\rm GL(n, \mathbb{Z})$.

\begin{proposition}
The representation $\phi: E\Pi A_n \to H_n^2$ given by $\mu_{ij} \mapsto t_{ij}(a)$ is not faithful for all $n \geq 3$.
\end{proposition}

\begin{proof}
It is enough to prove the result for $n=3$. Observe that $\mu_{23} \mu_{21} \neq \mu_{21} \mu_{23}$ in $E \Pi A_3$. But $t_{23}(2) t_{21}(2)=t_{21}(2) t_{23}(2)$ in ${\rm SL}(3, \Z)$. Hence $\phi$ is not faithful.
\end{proof}

The following problem seems interesting.

\begin{problem}
Describe the kernel of the map $\phi$.
\end{problem}
\bigskip

Let $IA_n$ denote the group of those automorphisms of $F_n$ that induce identity on the abelianization of $F_n$. Then we have the following
short exact sequence 
\begin{equation} \label{se1} 1 \to IA_n \to Aut(F_n) \stackrel{\psi}{\to} {\rm GL}(n,  \mathbb{Z}) \to 1.\end{equation}

As remarked in \cite{Collins} (see also \cite{Glover}), the image of $\Pi A_n$ in $\rm GL(n, \mathbb{Z})$ is the subgroup of $\rm GL(n, \mathbb{Z})$
consisting of invertible matrices where each column has exactly one odd entry and the rest are even. The subgroup is the semidirect product
$\Gamma_2(\mathbb{Z}) \rtimes S_n$, where $S_n$ is the symmetric group on $n$ symbols and $\Gamma_2(\mathbb{Z})$ is the level 2-congruence subgroup defined by the short exact sequence
$$1 \to \Gamma_2(\mathbb{Z}) \to \rm GL(n, \mathbb{Z}) \to {\rm GL}(n,  \mathbb{Z}/2) \to 1.$$
It is well known (see for example \cite{Fullarton}) that $\Gamma_2(\mathbb{Z})$ is generated by transvections $t_{ij}(2)$ and the diagonal matrices
$d_i$ for $1 \leq i \leq n $ which differ from the identity only in having $-1$ in the $(i,i)$th position.

It is easy to see that the matrix of the automorphism induced by $\mu_{ij}$ on $\mathbb{Z}^n$ with respect to the standard basis of $\mathbb{Z}^n$ is $t_{ij}(2)$ for all $1 \leq i \neq j \leq n$. Thus, $\psi \big(E \Pi A_n\big)=H_n^2$ for all $n \geq 2$. This can be thought of as a generalization of Sanov representation. The case $n=2$ gives the standard Sanov representation.

\begin{proposition}
$\psi|_{E\Pi A_2}: E\Pi A_2 \to H_2^2$ is an isomorphism.
\end{proposition}

\begin{proof}
By Sanov's theorem $H_2^2=\langle t_{12}(2), t_{21}(2) \rangle \cong F_2$. Thus $\psi|_{E\Pi A_2}$ is an epimorphism onto $F_2$. But $E\Pi A_2 \cong F_2$ and $F_2$ being Hopfian implies that $\psi|_{E\Pi A_2}$ is an isomorphism.
\end{proof}

This yields another proof of Proposition \ref{repfaith}.
\bigskip

\section{Residual nilpotence of $\Pi A_n$}\label{sec5}
In this section, we investigate the residual nilpotence of $\Pi A_n$ for $n \geq 2$. In \cite{ne2}, Nekritsuhin proved that $\Pi A_2$ is residually nilpotent. We prove the following.

\begin{proposition}
$\Pi A_n$ is not residually nilpotent for $n \geq 3$.
\end{proposition}

\begin{proof}
By the result of Collins, we have $\Pi A_n= E\Pi A_n \rtimes ES_n$. In particular, $\Pi A_n$ contains the symmetric group $S_3$ for $n \geq 3$. But $S_3$ is not residually nilpotent, since $\gamma_1(S_3)=S_3$ and $\gamma_r(S_3)=A_3$ for $r \geq 2$. Hence $\Pi A_n$ is not residually nilpotent for $n \geq 3$.
\end{proof}

It is interesting to investigate residual nilpotence of $E \Pi A_n$ for $n \geq 3$. Let $t_{ij}=t_{ij}(2)$ for $1 \leq i \neq j \leq n$,
$$U_n= \langle t_{ij}~|~ 1 \leq i < j \leq n\rangle$$
and
$$L_n= \langle t_{ij}~|~ 1 \leq j < i \leq n\rangle.$$ Then $U_n \cong L_n$ and  $H_n = \langle U_n,L_n \rangle$. Let $R$ be the set of defining relations of $E \Pi A_n$. Let $B_n =\{ \mu_{ij} \ | \ 1\leq i < j \leq n\}$ and $C_n =\{ \mu_{ij} \ | \ 1\leq j< i \leq n\}$. Using these, we define
$$\widetilde{ U_n} =\langle B_n~||~ R\rangle~\textrm{and}~\widetilde{ L_n} =\langle C_n~||~ R\rangle.$$

Specializing to the case $n=3$ gives $$\widetilde{U_3}= \langle \mu_{12}, \mu_{13}, \mu_{23}~||~ [\mu_{13}, \mu_{23}]=1, \mu_{13}\mu_{23}\mu_{12}= \mu_{12}\mu_{23}\mu_{13}^{-1}\rangle$$
and
$$\widetilde{L_3}= \langle \mu_{21}, \mu_{31}, \mu_{32}~||~ [\mu_{21}, \mu_{31}]=1, \mu_{21}\mu_{31}\mu_{23}= \mu_{23}\mu_{31}\mu_{21}^{-1}\rangle.$$
Further observe that $ \widetilde{U_3} \cong  \widetilde{ L_3}$ and $E \Pi A_3 = \langle \widetilde{U_3}, \widetilde{ L_3}\rangle$.

\begin{proposition}
$\widetilde{U_3}/ \langle [\mu_{12}, \mu_{13}] \rangle \cong U_3$ and $\widetilde{L_3}/ \langle [\mu_{31}, \mu_{32}] \rangle \cong L_3$.
\end{proposition}

\begin{proof}
First note that, we have $[t_{13}, t_{23}]=1$, $[t_{12}, t_{13}]=1$ and $[t_{12},t_{23}]=t_{13}^2$ in $U_3$. Also note that, $\mu_{13}\mu_{23}\mu_{12}= \mu_{12}\mu_{23}\mu_{13}^{-1}$ is equivalent to the relation $\mu_{13}[\mu_{13},\mu_{12}]\mu_{13}=[\mu_{12},\mu_{23}]$. It is now clear that the map $\phi$ induces the desired isomorphism $\widetilde{U_3}/ \langle [\mu_{12}, \mu_{13}] \rangle \cong U_3$. The other isomorphism can be proved analogously.
\end{proof}

\begin{theorem}
$\widetilde{U_3}$ and $\widetilde{L_3}$ are residually nilpotent.
\end{theorem}

\begin{proof}
Renaming the generators as $a=\mu_{13}\mu_{23}$, $b=\mu_{13}$ and $c=\mu_{12}$, we can write $$\widetilde{U_3}= \langle a,b,c~||~ [a,b]=1, c^{-1}ac=b^{-2}a\rangle.$$ Let $A=\langle a \rangle$ and $B=\langle b \rangle$. Then $\widetilde{U_3}$ is an HNN extension of $A \times B$ by $c$ and relative to the isomorphism $\phi: A \to B$ given by $\phi(a)= b^{-2}a$. Let $M_0= A \cap B$, $M_1= \phi^{-1}(M_0) \cap M_0 \cap  \phi(M_0)$ and inductively $$M_{i+1}= \phi^{-1}(M_i) \cap M_i \cap  \phi(M_i).$$ Since $M_0=1$, it follows that $H= \cap_{i=1}^{\infty}M_i=1$ and hence $\phi(H)=H$. As $A$ and $B$ are proper torsion free subgroups of $A\times B$, a theorem of Raptis and Varsos \cite[Theorem 11]{rv} implies that $\widetilde{U_3}$ is residually nilpotent. Similarly, $\widetilde{L_3}$ is residually nilpotent.
\end{proof}

We suspect that the following is true.

\begin{conjecture}
$\widetilde{U_n}$ and $\widetilde{L_n}$ are residually nilpotent for $n\geq 3$.
\end{conjecture}
\bigskip

\section{Structure of the Palindromic Torelli Group}\label{sec6}
In \cite{Collins}, Collins observed that the intersection $P I_n=IA_n \cap \Pi A_n$ is non-trivial for $n \geq 3$. For example, the commutator $\null [\mu_{12}, \mu_{13}]$ lies in this intersection. Recently, Fullarton proved the following interesting result.

\begin{theorem}\cite[Theorem A.]{Fullarton} 
The group $P I _n= IA_n \cap \Pi A_n$ is normally generated in $\Pi A_n$  by the automorphisms $\null [\mu_{12}, \mu_{13}]$ and $(\mu_{23}\mu_{13}^{-1}\mu_{31}\mu_{32}\mu_{12}\mu_{21}^{-1})^2$.
\end{theorem}

In \cite[p.5 ]{Fullarton}, Fullarton also observed that $PI_n = IA_n \cap \Pi A_n'$ for each $n \geq 3$. In this section, we prove the following theorem which provides a precise description of the palindromic Torelli group. 

\begin{theorem}\label{ptg}
$PI_n = IA_n \cap E \Pi A_n'$ for each $n \geq 3$.
\end{theorem}

To prove this theorem, we first find generators of $E \Pi A_3'$, and then identify those generators which lie in $IA_3$.

Recall that $E \Pi A_3= \langle A~||~R\rangle$, where
$$
A= \{\mu_{12},\mu_{21},\mu_{13},\mu_{23},\mu_{31},\mu_{32} \}
$$
and $R$ contains the following six long relations

$$ \mu_{12} \mu_{32} \mu_{13} = \mu_{13} \mu_{32} \mu_{12}^{-1},~~\mu_{21} \mu_{31} \mu_{23} = \mu_{23} \mu_{31} \mu_{21}^{-1},$$

$$ \mu_{13} \mu_{23} \mu_{12} = \mu_{12} \mu_{23} \mu_{13}^{-1},~~\mu_{23} \mu_{13} \mu_{21} = \mu_{21} \mu_{13} \mu_{23}^{-1},$$

$$\mu_{31} \mu_{21} \mu_{32} = \mu_{32} \mu_{21} \mu_{31}^{-1},~~\mu_{32} \mu_{12} \mu_{31} = \mu_{31} \mu_{12} \mu_{32}^{-1}, $$
and the following three relations of commutativity
$$
[\mu_{31},  \mu_{21}] = [\mu_{32}, \mu_{12}] = [\mu_{23},  \mu_{13}] = 1.
$$

Also, recall that $ES_3=\langle t_1, t_2, t_3, \alpha_{12}, \alpha_{23} \rangle$, and $ES_3$ action on $\Pi A_3$ yield the following
$$t_1 \mu_{12} t_1=\mu_{12}^{-1}, ~ t_1 \mu_{13} t_1=\mu_{13}^{-1}, ~ t_1 \mu_{21} t_1=\mu_{21}^{-1},$$
$$t_1 \mu_{23} t_1=\mu_{23}, ~ t_1 \mu_{32} t_1=\mu_{32}, ~ t_1 \mu_{31} t_1=\mu_{31}^{-1},$$
$$ \alpha_{12} \mu_{12} \alpha_{12}=\mu_{21}, ~ \alpha_{12} \mu_{21} \alpha_{12} =\mu_{12}, ~ \alpha_{12} \mu_{13} \alpha_{12}=\mu_{23}.$$

For $\mu_{ij} \in E \Pi A_3$, let $\widetilde{\mu_{ij}}$ denote its image in $E \Pi A_3/E \Pi A_3'$. It is easy to see that
\begin{equation*}
\begin{split}
E \Pi A_3/E \Pi A_3'= ~& \langle \widetilde{\mu_{ij}}~||~[\widetilde{\mu_{ij}},\widetilde{\mu_{kl}}]=1, \widetilde{\mu_{ij}}^2=1\rangle\\
\cong ~& \underbrace{\mathbb{Z}/2 \oplus  \cdots \oplus \mathbb{Z}/2}_{6~copies}.
\end{split}
\end{equation*}
Hence $E \Pi A_3'$ is a finitely presented group. Let $$\Lambda_3=\{\mu_{12}^{\epsilon_1}\mu_{21}^{\epsilon_2}\mu_{13}^{\epsilon_3}\mu_{23}^{\epsilon_4}\mu_{31}^{\epsilon_5}\mu_{32}^{\epsilon_6}~|~ \epsilon_i=0,1 \}$$ be a Schreier set of coset representatives of $E \Pi A_3'$ in $E \Pi A_3$. Consider the map $$E \Pi A_3 \to \Lambda_3$$ mapping $w \in E \Pi A_3$ to its representative $\widetilde{w} \in \Lambda_3$. Then by \cite[Theorem 2.7]{Magnus}, $E \Pi A_3'$ is generated by the set $$\{S_{\lambda, a}= (\lambda a)(\widetilde{ \lambda a})^{-1}~|~ \lambda \in \Lambda_3~\textrm{and}~ a \in A \}.$$

For $\lambda= \mu_{12}^{\epsilon_1}\mu_{21}^{\epsilon_2}\mu_{13}^{\epsilon_3}\mu_{23}^{\epsilon_4}\mu_{31}^{\epsilon_5}\mu_{32}^{\epsilon_6} \in \Lambda_3$, we set $$|\lambda|=\epsilon_1+ \cdots+\epsilon_6.$$ Further, we define an order on the generators of $E \Pi A_3$ by the rule $$\mu_{12}<\mu_{21}<\mu_{13}<\mu_{23}<\mu_{31}<\mu_{32}.$$
We use this ordering and induction on $|\lambda|$ to find the set of generators of $E \Pi A_3'$. Clearly, $S_{\lambda, a}=1$ for $|\lambda|=0$.

\begin{proposition}
For $1\leq |\lambda| \leq 6$, the generators of $E \Pi A_3'$ are given as follows:
\begin{enumerate}

\item If $|\lambda|=1$, then
\begin{displaymath}
S_{\mu_{ij}, \mu_{kl}} = \left\{ \begin{array}{ll}
1 & \textrm{if $\mu_{ij}< \mu_{kl}$}\\
\mu_{ij}^2 & \textrm{if $\mu_{ij}=\mu_{kl}$}\\
\null [ \mu_{ij}^{-1}, \mu_{kl}^{-1} ] & \textrm{if $\mu_{kl}<\mu_{ij}.$}
\end{array} \right.
\end{displaymath}

\item If $|\lambda|=2$, then
\begin{displaymath}
S_{\mu_{i_1j_1}\mu_{i_2j_2}, \mu_{kl}} = \left\{ \begin{array}{ll}
1 & \textrm{if $\mu_{i_2j_2}< \mu_{kl}$}\\
(\mu_{kl}^2)^{\mu_{i_1j_1}^{-1}} & \textrm{if $\mu_{i_2j_2}=\mu_{kl}$}\\
\null [\mu_{i_2j_2}^{-1},\mu_{kl}^{-1}]^{\mu_{i_1j_1}^{-1}} & \textrm{if $\mu_{i_1j_1} < \mu_{kl}<\mu_{i_2j_2}$}\\
\mu_{i_1j_1}\mu_{i_1j_1}^{\mu_{i_2 j_2}^{-1}} & \textrm{if $\mu_{i_1j_1}=\mu_{kl}$}\\
\null [\mu_{i_2j_2}^{-1}\mu_{i_1j_1}^{-1}, \mu_{kl}^{-1}] & \textrm{if $\mu_{kl}<\mu_{i_1j_1}$.}\\
\end{array} \right.
\end{displaymath}

\item If $|\lambda|=3$, then
\begin{displaymath}
S_{\mu_{i_1j_1}\mu_{i_2j_2}\mu_{i_3j_3}, \mu_{kl}} = \left\{ \begin{array}{ll}
1 & \textrm{if $\mu_{i_3j_3}< \mu_{kl}$}\\
(\mu^2_{i_3j_3})^{\mu_{i_2j_2}^{-1}\mu_{i_1j_1}^{-1}} & \textrm{if $\mu_{i_3j_3}=\mu_{kl}$}\\ \\
\null [\mu_{i_3j_3}^{-1},\mu_{kl}^{-1}]^{\mu_{i_2j_2}^{-1} \mu_{i_1j_1}^{-1}} & \textrm{if $\mu_{i_2j_2} < \mu_{kl} < \mu_{i_3j_3}$}\\ \\
(\mu_{i_2j_2} \mu_{i_2j_2}^{\mu_{i_3j_3}^{-1}})^{\mu_{i_1j_1}^{-1}} & \textrm{if $\mu_{i_2j_2}=\mu_{kl}$}\\ \\
\null [ \mu_{i_3j_3}^{-1} \mu_{i_2j_2}^{-1}, \mu_{kl}^{-1} ]^{\mu_{i_1j_1}^{-1}} & \textrm{if $\mu_{i_1j_1} < \mu_{kl} < \mu_{i_2j_2}$}\\ \\
\mu_{i_1j_1} \mu_{i_1j_1}^{\mu_{i_3j_3}^{-1} \mu_{i_2j_2}^{-1}} & \textrm{if $\mu_{i_1j_1}=\mu_{kl}$}\\ \\
\null [\mu_{i_3j_3}^{-1}\mu_{i_2j_2}^{-1} \mu_{i_1j_1}^{-1}, \mu_{kl}^{-1}] & \textrm{if $\mu_{kl}<\mu_{i_1j_1}$.}\\
\end{array} \right.
\end{displaymath}
\item If $|\lambda|=4$, then
\begin{displaymath}
S_{\mu_{i_1j_1}\mu_{i_2j_2}\mu_{i_3j_3}\mu_{i_4j_4}, \mu_{kl}} = \left\{ \begin{array}{ll}
1 & \textrm{if $\mu_{i_4j_4}< \mu_{kl}$}\\
(\mu^2_{i_4j_4})^{\mu_{i_3j_3}^{-1}\mu_{i_2j_2}^{-1}\mu_{i_1j_1}^{-1}} & \textrm{if $\mu_{i_4j_4}=\mu_{kl}$}\\ \\
\null [\mu_{i_4j_4}^{-1},\mu_{kl}^{-1}]^{\mu_{i_3j_3}^{-1} \mu_{i_2j_2}^{-1} \mu_{i_1j_1}^{-1}} & \textrm{if $\mu_{i_3j_3} < \mu_{kl} < \mu_{i_4j_4}$}\\ \\
(\mu_{i_3j_3}\mu_{i_3j_3}^{\mu_{i_4j_4}^{-1}})^{\mu_{i_2j_2}^{-1}\mu_{i_1j_1}^{-1}} & \textrm{if $\mu_{i_3j_3}=\mu_{kl}$}\\ \\
 \null [\mu_{i_4j_4}^{-1}\mu_{i_3j_3}^{-1},\mu_{kl}^{-1}]^{\mu_{i_2j_2}^{-1} \mu_{i_1j_1}^{-1}} & \textrm{if $\mu_{i_2j_2} < \mu_{kl} < \mu_{i_3j_3}$}\\ \\
(\mu_{i_2j_2} \mu_{i_2j_2}^{\mu_{i_4j_4}^{-1}\mu_{i_3j_3}^{-1}})^{\mu_{i_1j_1}^{-1}} & \textrm{if $\mu_{i_2j_2}=\mu_{kl}$}\\ \\
\null [\mu_{i_4j_4}^{-1} \mu_{i_3j_3}^{-1} \mu_{i_2j_2}^{-1}, \mu_{kl}^{-1} ]^{\mu_{i_1j_1}^{-1}} & \textrm{if $\mu_{i_1j_1} < \mu_{kl} < \mu_{i_2j_2}$}\\ \\
\mu_{i_1j_1} \mu_{i_1j_1}^{\mu_{i_4j_4}^{-1}\mu_{i_3j_3}^{-1} \mu_{i_2j_2}^{-1}} & \textrm{if $\mu_{i_1j_1}=\mu_{kl}$}\\ \\
 \null [\mu_{i_4j_4}^{-1}\mu_{i_3j_3}^{-1}\mu_{i_2j_2}^{-1} \mu_{i_1j_1}^{-1}, \mu_{kl}^{-1}] & \textrm{if $\mu_{kl}<\mu_{i_1j_1}$.}\\
\end{array} \right.
\end{displaymath}

\item If $|\lambda|=5$, then
\begin{displaymath}
S_{\mu_{i_1j_1}\mu_{i_2j_2}\mu_{i_3j_3}\mu_{i_4j_4}\mu_{i_5j_5}, \mu_{kl}} = \left\{ \begin{array}{ll}
1 & \textrm{if $\mu_{i_5j_5}< \mu_{kl}$}\\
(\mu^2_{i_5j_5})^{\mu_{i_4j_4}^{-1}\mu_{i_3j_3}^{-1}\mu_{i_2j_2}^{-1}\mu_{i_1j_1}^{-1}} & \textrm{if $\mu_{i_5j_5}=\mu_{kl}$}\\ \\
\null [\mu_{i_5j_5}^{-1},\mu_{kl}^{-1}]^{\mu_{i_4j_4}^{-1}\mu_{i_3j_3}^{-1} \mu_{i_2j_2}^{-1} \mu_{i_1j_1}^{-1}} &
\textrm{if $\mu_{i_4j_4} < \mu_{kl} < \mu_{i_5j_5}$}\\ \\
(\mu_{i_4j_4} \mu_{i_4j_4}^{\mu_{i_5j_5}^{-1}})^{\mu_{i_3j_3}^{-1}\mu_{i_2j_2}^{-1}\mu_{i_1j_1}^{-1}} & \textrm{if $\mu_{i_4j_4}=\mu_{kl}$}\\ \\
\null [\mu_{i_5j_5}^{-1} \mu_{i_4j_4}^{-1},\mu_{kl}^{-1}]^{ \mu_{i_3j_3}^{-1} \mu_{i_2j_2}^{-1} \mu_{i_1j_1}^{-1}} & \textrm{if $\mu_{i_3j_3} < \mu_{kl} < \mu_{i_4j_4}$}\\ \\
(\mu_{i_3j_3} \mu_{i_3j_3}^{ \mu_{i_5j_5}^{-1} \mu_{i_4j_4}^{-1}})^{\mu_{i_2j_2}^{-1}\mu_{i_1j_1}^{-1}} & \textrm{if $\mu_{i_3j_3}=\mu_{kl}$}\\ \\
\null [\mu_{i_5j_5}^{-1} \mu_{i_4j_4}^{-1}\mu_{i_3j_3}^{-1},\mu_{kl}^{-1}]^{\mu_{i_2j_2}^{-1} \mu_{i_1j_1}^{-1}} & \textrm{if $\mu_{i_2j_2} < \mu_{kl} < \mu_{i_3j_3}$}\\ \\
(\mu_{i_2j_2} \mu_{i_2j_2}^{\mu_{i_5j_5}^{-1} \mu_{i_4j_4}^{-1}\mu_{i_3j_3}^{-1}})^{\mu_{i_1j_1}^{-1}} & \textrm{if $\mu_{i_2j_2}=\mu_{kl}$}\\ \\
\null [\mu_{i_5j_5}^{-1} \mu_{i_4j_4}^{-1} \mu_{i_3j_3}^{-1} \mu_{i_2j_2}^{-1}, \mu_{kl}^{-1} ]^{\mu_{i_1j_1}^{-1}} & \textrm{if $\mu_{i_1j_1} < \mu_{kl} < \mu_{i_2j_2}$}\\ \\
\mu_{i_1j_1} \mu_{i_1j_1}^{\mu_{i_5j_5}^{-1} \mu_{i_4j_4}^{-1}\mu_{i_3j_3}^{-1} \mu_{i_2j_2}^{-1}} & \textrm{if $\mu_{i_1j_1}=\mu_{kl}$}\\ \\
\null [\mu_{i_5j_5}^{-1} \mu_{i_4j_4}^{-1}\mu_{i_3j_3}^{-1}\mu_{i_2j_2}^{-1} \mu_{i_1j_1}^{-1}, \mu_{kl}^{-1}] & \textrm{if $\mu_{kl}<\mu_{i_1j_1}$.}\\
\end{array} \right.
\end{displaymath}

\item If $|\lambda|=6$, then
\begin{displaymath}
S_{\mu_{12}\mu_{21}\mu_{13}\mu_{23}\mu_{31}\mu_{32}, \mu_{kl}} = \left\{ \begin{array}{ll}
(\mu^2_{32})^{\mu_{31}^{-1}\mu_{23}^{-1}\mu_{13}^{-1}\mu_{21}^{-1}\mu_{12}^{-1}} & \textrm{if $\mu_{32}=\mu_{kl}$}\\ \\
(\mu_{31} \mu_{31}^{\mu_{32}^{-1}})^{\mu_{23}^{-1}\mu_{13}^{-1}\mu_{21}^{-1}\mu_{12}^{-1}} & \textrm{if $\mu_{31}=\mu_{kl}$}\\ \\
(\mu_{23} \mu_{23}^{\mu_{32}^{-1} \mu_{31}^{-1}})^{\mu_{13}^{-1}\mu_{21}^{-1}\mu_{12}^{-1}} & \textrm{if $\mu_{23}=\mu_{kl}$}\\ \\
(\mu_{13} \mu_{13}^{\mu_{32}^{-1} \mu_{31}^{-1} \mu_{23}^{-1}})^{\mu_{21}^{-1}\mu_{12}^{-1}} & \textrm{if $\mu_{13}=\mu_{kl}$}\\ \\
(\mu_{21} \mu_{21}^{\mu_{32}^{-1} \mu_{31}^{-1} \mu_{23}^{-1} \mu_{13}^{-1}})^{\mu_{12}^{-1}} & \textrm{if $\mu_{21}=\mu_{kl}$}\\ \\
\mu_{12} \mu_{12}^{\mu_{32}^{-1} \mu_{31}^{-1} \mu_{23}^{-1} \mu_{13}^{-1}\mu_{21}^{-1}} & \textrm{if $\mu_{21}=\mu_{kl}$}.\\ \\
\end{array} \right.
\end{displaymath}

\end{enumerate}
\end{proposition}

\begin{proof}
The proof follows from the definition of $S_{\lambda, a}$ and are omitted.
\end{proof}

Thus, we have found  a generating set for the commutator subgroup $E \Pi A_3'$. Next, we identify generators of $E \Pi A_3'$ which lie in $IA_3$. For this, we introduce some functions as follows.

If $u \in F_n$ and $1 \leq j \leq n$, then we define $$l_j(u)= \textrm{sum of powers with which the generator}~ x_j~\textrm{appears in}~ u.$$ For example, if $u = x_3^2 x_1^{-2} x_3^5 x_1^2$, then
$$
l_1(u) = 0,~~l_2(u) = 0,~~l_3(u) = 7.
$$
We can see that for each $u, v \in F_n$ and integer $k$ the following formulas hold
$$
l_j(u v) = l_j(u) + l_j(v),~~~l_j(u^k) = k l_j(u).
$$
For $\varphi \in Aut(F_n)$ and $1 \leq i, j \leq n$, we define
$$
L_i^j (\varphi) = l_j \big((x_i)\varphi\big).
$$

It is not difficult to prove the following result.

\begin{lemma}
\begin{enumerate}
\item[]
\item If $\varphi \in IA_n$, then $L_i^i (\varphi) = 1$ and $L_i^j (\varphi) = 0$ for $i\not=j$.
\item If $\varphi \in E \Pi A_n$, then $L_i^i (\varphi) \equiv 1~ (mod~ 2)$ and $L_i^j (\varphi) \equiv 0~ (mod~ 2)$ for $i\not=j$.
\end{enumerate}
\end{lemma}

\begin{proof}
(1) First note that, if $u \in F_n'$, then $l_i(u) = 0$ for each $i$. By definition, $\varphi$ is an $IA$-automorphism of $F_n$ if and only if $(x_i)\varphi = x_i u_i$ for some $u_i \in F_n'$ and for each generator $x_i$. Therefore, if $\varphi \in IA_n$, then
$$L_i^i(\varphi) = l_i\big((x_i)\varphi\big) = l_i(x_i u_i) = l_i(x_i) + l_i(u_i) = 1$$
and  $$ L_i^j(\varphi) = l_j\big((x_i)\varphi\big) = l_j(x_i u_i) = l_j(x_i) + l_j(u_i) = 0$$ for $i\not= j$. 

(2) By definition, $\varphi \in E \Pi A_n$  if and only if $(x_i)\varphi = u_i x_i \overline{u_i}$ for some $u_i \in F_n$ and for each generator $x_i$. Here $\overline{u_i}$ is the reverse of $u_i$. Therefore, if $\varphi \in E \Pi A_n$, then
$$ L_i^i(\varphi) = l_i\big((x_i)\varphi\big) = l_i(u_i x_i \overline{u_i}) = l_i(x_i) + 2 l_i(u_i) \equiv 1~(mod~ 2)$$
and
$$ L_i^j(\varphi) = l_j\big((x_i)\varphi\big) = l_j(u_i x_i \overline{u_i}) =  2 l_j(u_i) \equiv 0~(mod~ 2)$$ for $i\not= j$. This proves the lemma.
\end{proof}

The following observation will be used quite often.

\begin{lemma}\label{lemconj}
Let $\phi \in E \Pi A_3'$. Then $\phi \in IA_3$ if and only if $\beta^{-1} \phi \beta \in IA_3$ for all $\beta \in  \Pi A_3$.
\end{lemma}
\begin{proof}
The proof follows from the fact that $IA_3$ is a normal subgroup of $Aut(F_3)$.
\end{proof}

Recall that, $t_{ij}(2)$ is the matrix of the automorphism induced by $\mu_{ij}$ on the abelianisation of $F_3$ with respect to the standard
basis of $\mathbb{Z}^3$. For an automorphism $\phi \in E \Pi A_3'$, let $t_{\phi}$ denote the product of matrices corresponding to the generators  appearing in $\phi$. Further, let 
$$L_1(\phi)=\begin{pmatrix} L_1^1(\phi)~~~ L_1^2(\phi) ~~ L_1^3(\phi)\end{pmatrix},$$
$$L_2(\phi)=\begin{pmatrix} L_2^1(\phi) ~~L_2^2(\phi) ~~ L_2^3(\phi)\end{pmatrix},$$
$$L_3(\phi)=\begin{pmatrix} L_3^1(\phi) ~~L_3^2(\phi) ~~ L_3^3(\phi)\end{pmatrix}.$$
Then it is easy to see that 
$$L_1(\phi) = \begin{pmatrix} 1 ~~ 0~~0  \end{pmatrix} t_{\phi},$$
$$L_2(\phi) = \begin{pmatrix} 0 ~~1~~0  \end{pmatrix}t_{\phi},$$
$$L_3(\phi) = \begin{pmatrix} 0~~0~~1  \end{pmatrix}t_{\phi}.$$

We use this observation and the above lemmas to identify generators of $E \Pi A_3'$ which lie in $IA_3$. First of all, note that, conjugates of all the generators of $E \Pi A_3'$ corresponding to the cases $|\lambda| =1,2,3,4, 6$ (except $\mu_{12} \mu_{12}^{\mu_{32}^{-1} \mu_{31}^{-1} \mu_{23}^{-1} \mu_{13}^{-1}\mu_{21}^{-1}}$ corresponding to  $|\lambda| =6$)
appear as generators corresponding to the case $|\lambda| =5$. Further, direct computations yield that $L_2^1\big(\mu_{12} \mu_{12}^{\mu_{32}^{-1} \mu_{31}^{-1} \mu_{23}^{-1} \mu_{13}^{-1}\mu_{21}^{-1}}) \neq 0$ and hence this generator does not lie in $IA_3$. This together with Lemma \ref{lemconj} implies that it is enough to find generators corresponding to the case  $|\lambda| =5$ which lie in $IA_3$.

Next, we prove the following lemma, which is of independent interest, and will be used in the proof of \thmref{ptg}.

\begin{proposition}
A generator $S_{\lambda, a}$ of $E \Pi A_3'$ lies in $IA_3$ if and only if it is one of the following types:
\begin{enumerate}
\item $\null [\mu_{13}^{-1}, \mu_{12}^{-1}]$.
\item $\null [\mu_{23}^{-1}, \mu_{21}^{-1}]^{a^{-1}}$, where $a \in \{1, \mu_{12}\}$.
\item $\null [\mu_{32}^{-1}, \mu_{31}^{-1}]^{b^{-1}}$, where $b \in \{\mu_{12}^{\epsilon_1} \mu_{21}^{\epsilon_2} \mu_{13}^{\epsilon_3}\mu_{23}^{\epsilon_4}~|~ \epsilon_i=0,1\}$.
\end{enumerate}
\end{proposition}

\begin{proof}
Let $S_{\lambda, a}$ denote a generator corresponding to the case $|\lambda| =5$. In view of Lemma \ref{lemconj}, we can take $S_{\lambda, a}$ to be without any conjugation.

\textbf{(A) $\mu_{i_5j_5}=\mu_{kl}$.} In this case, $S_{\lambda, a}= \mu_{ij}^2$ and $L_i^j (S_{\lambda, a}) \neq 0$. Therefore, $S_{\lambda, a}$ is not an  $IA$-automorphism.

\textbf{(B) $ \mu_{i_4j_4} < \mu_{kl} < \mu_{i_5j_5}$.} For simplicity, take $S_{\lambda, a}= \null [ \mu_{ij}^{-1}, \mu_{kl}^{-1} ]$, where $ \mu_{kl} < \mu_{ij}$. We have the following two cases depending on the cardinality of the set $\{ i, j, k, l \}$.

\begin{enumerate}
\item If $|\{ i, j, k, l \}| = 2$, then we have the following two subcases:
\begin{enumerate}
\item If $i = k$ and $j = l$, then $[\mu_{ij}^{-1}, \mu_{ij}^{-1}] = 1$.
\item If $i = l$ and $j = k$, then $L_j^i ([\mu_{ij}^{-1}, \mu_{ji}^{-1}]) \neq 0$.
\item[] Therefore, this automorphism is not an $IA$-automorphism.
\end{enumerate}
\item If  $|\{ i, j, k, l \}| = 3$, then we have the following four subcases:
\begin{enumerate}
\item If $i = l$, then $L_k^j ([\mu_{ij}^{-1}, \mu_{ki}^{-1}]) \neq 0$. 
\item If $j = k$, then $L_i^l ([\mu_{ij}^{-1}, \mu_{jl}^{-1}]) \neq 0$.
\item[] Therefore, the above automorphisms are not $IA$-automorphisms.
\item If $j = l$, then $[\mu_{ij}^{-1}, \mu_{kj}^{-1}] = 1$.
\item If $i = k$, then $[\mu_{ij}^{-1}, \mu_{il}^{-1}]$ sends $x_i$ to $x_j x_l x_j^{-1} x_l^{-1} x_i x_l^{-1} x_j^{-1} x_l x_j$ and fixes all other generators. Therefore, $S_{\lambda, a}$ is an $IA$-automorphism.
\end{enumerate}
\end{enumerate}

\textbf{(C) $ \mu_{i_4j_4} = \mu_{kl} < \mu_{i_5j_5}$.} For simplicity, take $S_{\lambda, a}= \mu_{ij}\mu_{kl}\mu_{ij}\mu_{kl}^{-1}$. We have the following two cases depending on the cardinality of the set $\{ i, j, k, l \}$.

\begin{enumerate}
\item If $|\{ i,j,k,l \}| = 2$, then $S_{\lambda, a} =\mu_{ij}\mu_{ji}\mu_{ij}\mu_{ji}^{-1}$ and  $L_i^i(S_{\lambda, a}) \neq 1$. Therefore, $S_{\lambda, a}$ is not an $IA$-automorphism.
\item If $|\{ i,j,k,l \}| = 3$, then we have the following four subcases:
\begin{enumerate}
\item If $\mu_{ij} < \mu_{ik}$, then $L_i^j (S_{\lambda, a}) \neq 0$. 
\item If $\mu_{ij} < \mu_{ki}$, then $L_i^j (S_{\lambda, a}) \neq 0$.
\item If $\mu_{ij} < \mu_{j l}$, then $L_i^j (S_{\lambda, a}) \neq 0$.
\item If $\mu_{ij} < \mu_{k j}$, then $L_i^j (S_{\lambda, a}) \neq 0$. 
\item[] Therefore, $S_{\lambda, a}$ is not an $IA$-automorphism.
\end{enumerate}
\end{enumerate}

\textbf{(D)  $\mu_{i_3j_3} < \mu_{kl} < \mu_{i_4j_4}$.} Take $S_{\lambda, a}=\null [\mu_{i_5j_5}^{-1}\mu_{i_4j_4}^{-1}, \mu_{kl}^{-1}]$. Clearly, $|\{ i_4, j_4, i_5,j_5, k, l \}| = 3$. We have the following cases:

\begin{enumerate}
\item If $k=i_4$, then we have the following four subcases:
\begin{enumerate}
\item If $\mu_{kl}<\mu_{kj_4}< \mu_{l k}$, then $L_k^k (S_{\lambda, a}) \neq 1$.
\item If $\mu_{kl}<\mu_{kj_4}< \mu_{l j_4}$, then $L_{j_1}^k (S_{\lambda, a}) \neq 0$.
\item If $\mu_{kl}<\mu_{kj_4}< \mu_{j_4 k}$, then $L_l^k (S_{\lambda, a}) \neq 0$.
\item[] Therefore, $S_{\lambda, a}$ is not an $IA$ automorphism.
\item If $\mu_{kl}<\mu_{kj_4}< \mu_{j_4 l}$, then $S_{\lambda, a} = \null [\mu_{k j_4}^{-1}, \mu_{k l}^{-1}]$. Therefore, by (B)(2)(d),  $S_{\lambda, a}$ is an $IA$-automorphism.
\end{enumerate}
\item If $k=j_4$, then we have the following four subcases:
\begin{enumerate}
\item If $\mu_{kl}<\mu_{i_4 k}< \mu_{i_4 l}$, then $L_l^{i_1} (S_{\lambda, a}) \neq 0$.
\item If $\mu_{kl}<\mu_{i_4 k}< \mu_{k i_4}$, then $L_l^{i_1} (S_{\lambda, a}) \neq 0$.
\item If $\mu_{kl}<\mu_{i_4 k}< \mu_{l i_4}$, then $L_k^k (S_{\lambda, a}) \neq 1$. 
\item If $\mu_{kl}<\mu_{i_4k }< \mu_{l k}$, then $L_k^k (S_{\lambda, a}) \neq 1$.
\item[] Therefore, $S_{\lambda, a}$ is not an $IA$-automorphism.
\end{enumerate}
\item If $l=i_4$, then we have the following four subcases:
\begin{enumerate}
\item If $\mu_{kl}<\mu_{l j_4}< \mu_{l k}$, then $L_l^k (S_{\lambda, a}) \neq 0$. 
\item If $\mu_{kl}<\mu_{l j_4}< \mu_{k j_4}$, then $L_{j_1}^k (S_{\lambda, a}) \neq 0$.
\item If $\mu_{kl}<\mu_{l j_4}< \mu_{j_4 k}$, then $L_{j_1}^k (S_{\lambda, a}) \neq 0$.
\item If $\mu_{kl}<\mu_{l j_4}< \mu_{j_4 l}$, then $L_{j_1}^k (S_{\lambda, a}) \neq 0$.
\item[] Therefore, $S_{\lambda, a}$ is not an $IA$-automorphism.
\end{enumerate}
\item If $l=j_4$, then we have the following four subcases:
\begin{enumerate}
\item If $\mu_{kl}<\mu_{i_4 l}< \mu_{i_4 k}$, then $L_l^i (S_{\lambda, a}) \neq 0$.
\item If $\mu_{kl}<\mu_{i_4 l}< \mu_{l i_4}$, then $L_i^k (S_{\lambda, a}) \neq 0$. 
\item If $\mu_{kl}<\mu_{i_4 l}< \mu_{l k}$, then $S_{\lambda, a}= \null [\mu_{l k}^{-1}, \mu_{k l}^{-1}]^{\mu_{i_4 l}^{-1}}$. By (B)(1), $\null [\mu_{l k}^{-1}, \mu_{k l}^{-1}]$ is not an $IA$-automorphism.
\item[] Therefore, $S_{\lambda, a}$ is not an $IA$-automorphism.
\item If $\mu_{kl}<\mu_{i_4 l}< \mu_{k i_4}$, then $S_{\lambda, a}= \null [\mu_{k i_1}^{-1}, \mu_{k l}^{-1}]^{\mu_{i_1 l}^{-1}}$. By (B)(2)(d), $\null [\mu_{k i_1}^{-1}, \mu_{k l}^{-1}]$ is an $IA$-automorphism.
\end{enumerate}
\item If $k=j_4$ and $l=i_4$, then we have the following four subcases:
\begin{enumerate}
\item If $\mu_{kl}<\mu_{l k}< \mu_{k j_5}$, then $L_l^l (S_{\lambda, a}) \neq 1$.
\item If $\mu_{kl}<\mu_{l k}< \mu_{l j_5}$, then $L_{j_5}^k (S_{\lambda, a}) \neq 0$. 
\item If $\mu_{kl}<\mu_{l k}< \mu_{i_5 k}$, then $L_k^{i_5} (S_{\lambda, a}) \neq 0$. 
\item If $\mu_{kl}<\mu_{l k}< \mu_{i_5 l}$, then $L_l^l (S_{\lambda, a}) \neq 1$. 
\item[] Therefore, $S_{\lambda, a}$ is not an $IA$-automorphism.
\end{enumerate}
\end{enumerate}

\textbf{(E)  $\mu_{i_3 j_3} = \mu_{k l}$.} Take $S_{\lambda, a}=\mu_{i_3j_3}\mu_{i_3j_3}^{\mu_{i_5j_5}^{-1}\mu_{i_4j_4}^{-1}}$. Clearly, $|\{ i_3, j_3, i_4,j_4, i_5, j_5 \}| = 3$. We have the following cases:

\begin{enumerate}
\item If $i_3=i_4$, then we have the following four subcases:
\begin{enumerate}
\item If $\mu_{i_3 j_3}<\mu_{i_3j_4}< \mu_{j_3 i_3}$, then $L_{i_3}^{i_3} (S_{\lambda, a}) \neq 1$.
\item If $\mu_{i_3 j_3}<\mu_{i_3 j_4}< \mu_{j_3 j_4}$, then $L_{j_3}^{i_3} (S_{\lambda, a}) \neq 0$.
\item If $\mu_{i_3 j_3}<\mu_{i_3 j_4}< \mu_{j_4 i_3}$, then $L_{j_3}^{i_3} (S_{\lambda, a}) \neq 0$.
\item If $\mu_{i_3 j_3}<\mu_{i_3 j_4}< \mu_{j_4 j_3}$, then $L_{j_3}^{i_3} (S_{\lambda, a}) \neq 0$. 
\item[] Therefore, $S_{\lambda, a}$ is not an $IA$-automorphism.
\end{enumerate}

\item If $i_3=j_4$, then we have the following four subcases:
\begin{enumerate}
\item If $\mu_{i_3 j_3}<\mu_{i_4 i_3}< \mu_{i_3 i_4}$, then $L_{j_3}^{i_3} (S_{\lambda, a}) \neq 0$.
\item If $\mu_{i_3 j_3}<\mu_{i_4 i_3}< \mu_{i_4 j_3}$, then $\phi= \mu_{i_3 j_3} \mu_{i_3 j_3}^{\mu_{i_4 i_3}^{-1}}$. By (C), $S_{\lambda, a}$  is not an $IA$-automorphism.
\item If $\mu_{i_3 j_3}<\mu_{i_4 i_3}< \mu_{j_3 i_3}$, then $L_{j_3}^{i_3} (S_{\lambda, a}) \neq 0$.
\item If $\mu_{i_3 j_3}<\mu_{i_4 i_3}< \mu_{j_3 i_4}$, then $L_{j_3}^{i_3} (S_{\lambda, a}) \neq 0$. 
\item[] Therefore, $S_{\lambda, a}$ is not an $IA$-automorphism.
\end{enumerate}

\item If $j_3=i_4$, then we have the following four subcases:
\begin{enumerate}
\item If $\mu_{i_3 j_3}<\mu_{j_3 j_4}< \mu_{i_3 j_4}$, then $L_{j_3}^{i_3} (S_{\lambda, a}) \neq 0$.
\item If $\mu_{i_3 j_3}<\mu_{j_3 j_4}< \mu_{j_3 i_3}$, then $L_{i_3}^{i_3} (S_{\lambda, a}) \neq 1$.
\item If $\mu_{i_3 j_3}<\mu_{j_3 j_4}< \mu_{j_4 i_3}$, then $L_{j_3}^{i_3} (S_{\lambda, a}) \neq 0$.
\item If $\mu_{i_3 j_3}<\mu_{j_3 j_4}< \mu_{j_4 j_3}$, then $L_{j_3}^{i_3} (S_{\lambda, a}) \neq 0$. 
\item[] Therefore, $S_{\lambda, a}$ is not an $IA$-automorphism.
\end{enumerate}

\item If $j_3=j_4$, then we have the following four subcases:
\begin{enumerate}
\item If $\mu_{i_3 j_3}<\mu_{i_4 j_3}< \mu_{i_3 i_4}$, then $L_{j_3}^{i_3} (S_{\lambda, a}) \neq 0$.
\item If $\mu_{i_3 j_3}<\mu_{i_4 j_3}< \mu_{i_4 i_3}$, then $L_{j_3}^{i_3} (S_{\lambda, a}) \neq 0$.
\item If $\mu_{i_3 j_3}<\mu_{i_4 j_3}< \mu_{j_3 i_3}$, then $L_{i_3}^{i_3} (S_{\lambda, a}) \neq 1$.
\item If $\mu_{i_3 j_3}<\mu_{i_4 j_3}< \mu_{j_3 i_4}$, then $L_{i_4}^{i_3} (S_{\lambda, a}) \neq 0$. 
\item[] Therefore, $S_{\lambda, a}$ is not an $IA$-automorphism.
\end{enumerate}

\item If $i_3=j_4$ and $j_3=i_4$, then we have the following four subcases:
\begin{enumerate}
\item If $\mu_{i_3 j_3}<\mu_{j_3 i_3}< \mu_{i_3 j_5}$, then $L_{i_3}^{i_3} (\phi) \neq 1$.
\item If $\mu_{i_3 j_3}<\mu_{j_3 i_3}< \mu_{j_3 j_5}$, then $L_{j_5}^{i_3} (\phi) \neq 0$.
\item If $\mu_{i_3 j_3}<\mu_{j_3 i_3}< \mu_{i_5 i_3}$, then $L_{i_3}^{i_3} (\phi) \neq 1$.
\item If $\mu_{i_3 j_3}<\mu_{j_3 i_3}< \mu_{i_5 j_3}$, then $L_{i_3}^{i_3} (\phi) \neq 1$. 
\item[] Therefore, $S_{\lambda, a}$ is not an $IA$-automorphism.
\end{enumerate}
\end{enumerate}

\textbf{(F)  $\mu_{i_2j_2} < \mu_{kl} < \mu_{i_3j_3}$.} Take $S_{\lambda, a}=\null [\mu_{i_5j_5}^{-1} \mu_{i_4j_4}^{-1}\mu_{i_3j_3}^{-1},\mu_{kl}^{-1}]$. There are a total of 15 cases. When $j_5=l$, using the defining relations in $E \Pi A_n$, we see that $S_{\lambda, a}=\null [\mu_{i_4j_4}^{-1}\mu_{i_3j_3}^{-1}, \mu_{kl}^{-1}]$. By {\textbf (D)}, we see that none of them is an $IA$-automorphism. Now suppose that $j_5 \neq l$. Then we have the following eight cases. 
\begin{enumerate}
\item If $S_{\lambda, a}=[\mu_{23}^{-1} \mu_{13}^{-1} \mu_{21}^{-1} , \mu_{12}^{-1}]$, then $L^2_2(S_{\lambda, a}) \neq 1$. 
\item If $S_{\lambda, a}=[\mu_{31}^{-1} \mu_{23}^{-1} \mu_{21}^{-1} , \mu_{12}^{-1}]$, then $L^2_2(S_{\lambda, a}) \neq 1$. 
\item If $S_{\lambda, a}=[\mu_{31}^{-1} \mu_{13}^{-1} \mu_{21}^{-1} , \mu_{12}^{-1}]$, then $L^2_3(S_{\lambda, a}) \neq 0$. 
\item If $S_{\lambda, a}=[\mu_{31}^{-1} \mu_{23}^{-1} \mu_{13}^{-1} , \mu_{12}^{-1}]$, then $L^3_3(S_{\lambda, a}) \neq 1$. 
\item If $S_{\lambda, a}=[\mu_{32}^{-1} \mu_{31}^{-1} \mu_{13}^{-1} , \mu_{21}^{-1}]$, then $L^1_3(S_{\lambda, a}) \neq 0$. 
\item If $S_{\lambda, a}=[\mu_{32}^{-1} \mu_{23}^{-1} \mu_{13}^{-1} , \mu_{21}^{-1}]$, then $L^1_2(S_{\lambda, a}) \neq 0$.
\item If $S_{\lambda, a}=[\mu_{32}^{-1} \mu_{31}^{-1} \mu_{23}^{-1} , \mu_{21}^{-1}]$, then $L^1_2(S_{\lambda, a}) \neq 0$.
\item If $S_{\lambda, a}=[\mu_{32}^{-1} \mu_{31}^{-1} \mu_{23}^{-1} , \mu_{13}^{-1}]$, then $L^3_3(S_{\lambda, a}) \neq 1$.   
\item[] Therefore, $S_{\lambda, a}$ is not an $IA$-automorphism. 
\end{enumerate}

\textbf{(G) $\mu_{i_2j_2}= \mu_{kl}$.} Take $S_{\lambda, a}=\mu_{i_2j_2} \mu_{i_2j_2}^{\mu_{i_5j_5}^{-1} \mu_{i_4j_4}^{-1}\mu_{i_3j_3}^{-1}}$. We have the following cases.
\begin{enumerate}
\item If $\mu_{i_2j_2}=\mu_{12}$, then we have the following ten subcases:
\begin{enumerate}
\item If $S_{\lambda, a}=\mu_{12} \mu_{12}^{\mu_{23}^{-1} \mu_{13}^{-1}\mu_{21}^{-1}}$, then $L_1^1 (S_{\lambda, a}) \neq 1$.
\item If $S_{\lambda, a}=\mu_{12} \mu_{12}^{\mu_{31}^{-1} \mu_{13}^{-1}\mu_{21}^{-1}}$, then $L_1^1 (S_{\lambda, a}) \neq 1$.
\item If $S_{\lambda, a}=\mu_{12} \mu_{12}^{\mu_{31}^{-1} \mu_{23}^{-1}\mu_{21}^{-1}}$, then $L_3^3 (S_{\lambda, a}) \neq 1$.
\item If $S_{\lambda, a}=\mu_{12} \mu_{12}^{\mu_{31}^{-1} \mu_{23}^{-1}\mu_{21}^{-1}}$, then $L_3^3 (S_{\lambda, a}) \neq 1$.
\item If $S_{\lambda, a}=\mu_{12} \mu_{12}^{\mu_{31}^{-1} \mu_{23}^{-1}\mu_{13}^{-1}}$, then $L_3^3 (S_{\lambda, a}) \neq 1$.
\item[] Therefore, $S_{\lambda, a}$ is not an $IA$ automorphism.
\item If $S_{\lambda, a}=\mu_{12} \mu_{12}^{\mu_{32}^{-1} \mu_{23}^{-1}\mu_{21}^{-1}}=\mu_{12} \mu_{12}^{\mu_{23}^{-1}\mu_{21}^{-1}}$.
\item If $S_{\lambda, a}=\mu_{12} \mu_{12}^{\mu_{32}^{-1} \mu_{13}^{-1}\mu_{21}^{-1}}=\mu_{12} \mu_{12}^{\mu_{13}^{-1}\mu_{21}^{-1}}$.
\item If $S_{\lambda, a}=\mu_{12} \mu_{12}^{\mu_{32}^{-1} \mu_{23}^{-1}\mu_{13}^{-1}}=\mu_{12} \mu_{12}^{\mu_{23}^{-1}\mu_{13}^{-1}}$.
\item If $S_{\lambda, a}=\mu_{12} \mu_{12}^{\mu_{32}^{-1} \mu_{31}^{-1}\mu_{13}^{-1}}=\mu_{12} \mu_{12}^{\mu_{31}^{-1}\mu_{13}^{-1}}$.
\item If $S_{\lambda, a}=\mu_{12} \mu_{12}^{\mu_{32}^{-1} \mu_{31}^{-1}\mu_{23}^{-1}}=\mu_{12} \mu_{12}^{\mu_{31}^{-1}\mu_{23}^{-1}}$.
\item[] In the above cases, $S_{\lambda, a}$ is not an $IA$ automorphism by (E).
\end{enumerate}

\item  If $\mu_{i_2j_2}=\mu_{21}$, then we have the following four subcases:
\begin{enumerate}
\item If $S_{\lambda, a}=\mu_{21} \mu_{21}^{\mu_{32}^{-1} \mu_{23}^{-1}\mu_{13}^{-1}}$, then $L_3^3 (S_{\lambda, a}) \neq 1$.
\item If $S_{\lambda, a}=\mu_{21} \mu_{21}^{\mu_{32}^{-1} \mu_{31}^{-1}\mu_{13}^{-1}}$, then $L_1^2 (S_{\lambda, a}) \neq 0$.
\item If $S_{\lambda, a}=\mu_{21} \mu_{21}^{\mu_{32}^{-1} \mu_{31}^{-1}\mu_{23}^{-1}}$, then $L_1^3 (S_{\lambda, a}) \neq 0$.
\item[] Therefore, $S_{\lambda, a}$ is not an $IA$ automorphism.
\item If $S_{\lambda, a}=\mu_{21} \mu_{21}^{\mu_{31}^{-1} \mu_{23}^{-1}\mu_{13}^{-1}}= \mu_{21} \mu_{21}^{\mu_{23}^{-1}\mu_{13}^{-1}}$, then $S_{\lambda, a}$ is not an $IA$ automorphism by (E).
\end{enumerate}

\item If $\mu_{i_2j_2}=\mu_{13}$, then $S_{\lambda,a}= \mu_{13} \mu_{13}^{\mu_{32}^{-1} \mu_{31}^{-1}\mu_{23}^{-1}}$ and $L_1^1 (S_{\lambda, a}) \neq 1$. Therefore, $S_{\lambda, a}$ is not an $IA$ automorphism.
\end{enumerate}

\textbf{(H) $\mu_{i_1j_1} <\mu_{kl} < \mu_{i_2j_2}$.} Take $S_{\lambda, a}=\null [\mu_{i_5j_5}^{-1} \mu_{i_4j_4}^{-1} \mu_{i_3j_3}^{-1} \mu_{i_2j_2}^{-1}, \mu_{kl}^{-1} ]$. We have the following six cases:

\begin{enumerate}
\item If $S_{\lambda, a}= \null [\mu_{32}^{-1} \mu_{31}^{-1} \mu_{23}^{-1} \mu_{13}^{-1}, \mu_{21}^{-1} ]$, then $L_1^3 (S_{\lambda, a}) \neq 0$.
\item If $S_{\lambda, a}= \null [\mu_{32}^{-1} \mu_{31}^{-1} \mu_{23}^{-1} \mu_{13}^{-1}, \mu_{12}^{-1} ]$, then $L_{3}^{3} (S_{\lambda, a}) \neq 1$.
\item If $S_{\lambda, a}= \null [\mu_{32}^{-1} \mu_{31}^{-1} \mu_{23}^{-1} \mu_{21}^{-1}, \mu_{12}^{-1} ]$, then $L_{3}^{3} (S_{\lambda, a}) \neq 1$.
\item If $S_{\lambda, a}= \null [\mu_{32}^{-1} \mu_{31}^{-1} \mu_{13}^{-1} \mu_{21}^{-1}, \mu_{12}^{-1} ]$, then $L_{2}^{3} (S_{\lambda, a}) \neq 0$. 
\item If $S_{\lambda, a}= \null [\mu_{32}^{-1} \mu_{23}^{-1} \mu_{13}^{-1} \mu_{21}^{-1}, \mu_{12}^{-1} ]$, then $L_{2}^{2} (S_{\lambda, a}) \neq 1$. 
\item If $S_{\lambda, a}= \null [\mu_{32}^{-1} \mu_{23}^{-1} \mu_{13}^{-1} \mu_{21}^{-1}, \mu_{12}^{-1} ]$, then $L_{3}^{3} (S_{\lambda, a}) \neq 1$. 
\item[] Therefore, $S_{\lambda, a}$ is not an $IA$ automorphism.
\end{enumerate}

\textbf{(I)  $\mu_{i_1j_1}=\mu_{kl}$.} In this case, $S_{\lambda,a}=\mu_{i_1j_1} \mu_{i_1j_1}^{\mu_{i_5j_5}^{-1} \mu_{i_4j_4}^{-1}\mu_{i_3j_3}^{-1} \mu_{i_2j_2}^{-1}}$. We have the following six cases:
\begin{enumerate}
\item If $S_{\lambda,a}=\mu_{12} \mu_{12}^{\mu_{31}^{-1} \mu_{23}^{-1}\mu_{13}^{-1} \mu_{21}^{-1}}$,  then $L_3^1 (S_{\lambda,a}) \neq 0$.
\item If $S_{\lambda,a}=\mu_{12} \mu_{12}^{\mu_{32}^{-1} \mu_{23}^{-1}\mu_{13}^{-1} \mu_{21}^{-1}}$,  then $L_1^1 (S_{\lambda,a}) \neq 1$.
\item If $S_{\lambda,a}=\mu_{12} \mu_{12}^{\mu_{32}^{-1} \mu_{31}^{-1}\mu_{13}^{-1} \mu_{21}^{-1}}$,  then $L_1^1 (S_{\lambda,a}) \neq 1$.
\item If $S_{\lambda,a}=\mu_{12} \mu_{12}^{\mu_{32}^{-1} \mu_{31}^{-1}\mu_{23}^{-1} \mu_{21}^{-1}}$,  then $L_1^1 (S_{\lambda,a}) \neq 1$.
\item If $S_{\lambda,a}=\mu_{12} \mu_{12}^{\mu_{32}^{-1} \mu_{31}^{-1}\mu_{23}^{-1} \mu_{13}^{-1}}$,  then $L_3^2 (S_{\lambda,a}) \neq 0$.
\item If $S_{\lambda,a}=\mu_{21} \mu_{21}^{\mu_{32}^{-1} \mu_{31}^{-1}\mu_{23}^{-1} \mu_{13}^{-1}}$,  then $L_3^3 (S_{\lambda,a}) \neq 1$.
\item[] Therefore, $S_{\lambda, a}$ is not an $IA$ automorphism.
\end{enumerate}

\textbf{(J)  $\mu_{kl} < \mu_{i_1j_1}$.} In this case, $S_{\lambda, a}= \null [\mu_{32}^{-1}\mu_{31}^{-1}\mu_{23}^{-1}\mu_{13}^{-1}\mu_{21}^{-1}, \mu_{12}^{-1}]$ and $L_1^1 (S_{\lambda, a}) \neq 1$. Therefore, $S_{\lambda, a}$ is not an $IA$ automorphism.

Hence, the only generators  $S_{\lambda, a}$ that lie in $IA_3$ are $\null [\mu_{13}^{-1}, \mu_{12}^{-1}]$, $\null [\mu_{23}^{-1}, \mu_{21}^{-1}]^{a^{-1}}$ for $a \in \{1, \mu_{12}\}$ and $\null [\mu_{32}^{-1}, \mu_{31}^{-1}]^{b^{-1}}$ for $b \in \{\mu_{12}^{\epsilon_1} \mu_{21}^{\epsilon_2} \mu_{13}^{\epsilon_3}\mu_{23}^{\epsilon_4}~|~ \epsilon_i=0,1\}$. This completes the proof of the theorem.
\end{proof}

We now prove the main theorem of this section.\\

\textbf{Proof of \thmref{ptg}.}
Let $n \geq3$. Clearly, $IA_n \cap E \Pi A_n' \subset IA_n \cap \Pi A_n= PI_n$. In \cite{Fullarton}, Fullarton showed that $PI_n$ is normally generated in $\Pi A_n$ by the automorphisms $\null [\mu_{12}, \mu_{13}]$ and $\mu= \big(\mu_{23}\mu_{13}^{-1}\mu_{31}\mu_{32}\mu_{12}\mu_{21}^{-1}\big)^2$. It is easy to see that 

\begin{equation*}
\begin{split}
\mu = &~ S_{\mu_{13}\mu_{23},\mu_{13}}^{-1} S_{\mu_{13}\mu_{23}\mu_{31}\mu_{32}, \mu_{12}} S_{\mu_{12}\mu_{21}\mu_{13}\mu_{23} \mu_{31} \mu_{32}, \mu_{21}}^{-1} S_{\mu_{12}\mu_{21}\mu_{13}\mu_{23} \mu_{31} \mu_{32}, \mu_{23} } \\
~& S_{\mu_{12}\mu_{21}\mu_{31}\mu_{32}, \mu_{13}}^{-1}  S_{\mu_{12}\mu_{21}\mu_{31}\mu_{32}, \mu_{31}} S_{\mu_{12}\mu_{21}\mu_{32}, \mu_{32}} S_{\mu_{12}\mu_{21},\mu_{12}}.\\
\end{split}
\end{equation*}
It follows that $\null [\mu_{12}, \mu_{13}]$ and $\mu$ lie in $E \Pi A_3'$ and hence they also lie in $IA_n \cap E \Pi A_n'$. Note that, $IA_n$ is a normal subgroup of $Aut(F_n)$. Since   $E \Pi A_n'$ is a characteristic subgroup of $E \Pi A_n$ and  $E \Pi A_n$ is a normal subgroup of $\Pi A_n$, it follows that  $E \Pi A_n'$ is a normal subgroup of  $\Pi A_n$. Consequently, $\null [\mu_{12}, \mu_{13}]^{\sigma}$ and $\mu^{\sigma}$ lie in $IA_n \cap E \Pi A_n'$ for all $\sigma \in \Pi A_n$. Hence $IA_n \cap \Pi A_n \subset IA_n \cap E \Pi A_n'$. This proves the theorem. $\Box$

\bigskip

We conclude with the following questions, answers to which,  will shed more light on the structure of the palindromic automorphism group $\Pi A_n$ of $F_n$.

\begin{problem}
Is it possible to find defining relations for $E\Pi A_n'$?
\end{problem}

\begin{problem}
Is it possible to find a decomposition of $E\Pi A_n$ as a semi-direct product, a free product or an HNN extension of some groups?
\end{problem}

\begin{problem}
In view of Fullarton \cite{Fullarton1}, is there any geometrical or topological interpretation of Theorem \ref{ptg}?
\end{problem}

\medskip \noindent \textbf{Acknowledgement.} The authors are grateful to the anonymous referee for comments and suggestions which improved the paper. The authors gratefully acknowledge the support from the DST-RFBR project DST/INT/RFBR/P-137 and RFBR-13-01-92697. Bardakov is partially supported by Laboratory of Quantum Topology of Chelyabinsk State University via Russian Federation government grant 14.Z50.31.0020. Gongopadhyay is partially supported by NBHM grant NBHM/R. P.7/2013/Fresh/992. Singh is also supported by DST INSPIRE Scheme IFA-11MA-01/2011 and DST SERC Fast Track Scheme SR/FTP/MS-027/2010.

\end{document}